\documentclass[reqno, 12pt]{amsart}
\usepackage{amssymb}
\usepackage{amsfonts}
\usepackage{srcltx}

\newtheorem{theorem}{Theorem}[section]
\newtheorem{lemma}[theorem]{Lemma}

\theoremstyle{definition}
\newtheorem{definition}[theorem]{Definition}

\theoremstyle{remark}

\DeclareMathOperator{\diam}{diam}
\DeclareMathOperator{\dist}{dist}
\DeclareMathOperator{\conv}{conv}

\DeclareMathOperator{\st}{st}

\DeclareMathOperator{\Seq}{Seq}

\DeclareMathOperator{\gal}{gal}
\DeclareMathOperator{\sho}{^{\circ}}
\DeclareMathOperator{\sh}{sh}
\DeclareMathOperator{\mon}{mon}
\DeclareMathOperator{\cl}{cl}

\begin{document}
\title[]{Hyper-extensions in metric fixed point theory }
\author[]{Andrzej Wi\'{s}nicki}
\dedicatory{ Dedicated to Professor Sompong Dhompongsa on the occasion of
his 65th birthday}
\address{Andrzej Wi\'{s}nicki, Institute of Mathematics, Maria Curie-Sk{\l }%
odowska University, 20-031 Lublin, Poland}
\email{awisnic@hektor.umcs.lublin.pl}

\begin{abstract}
We apply a modern axiomatic system of nonstandard analysis in metric fixed
point theory. In particular, we formulate a nonstandard iteration scheme for
nonexpansive mappings and present a nonstandard approach to fixed-point
problems in direct sums of Banach spaces.
\end{abstract}

\date{}
\maketitle

\section{Introduction}

Nonstandard analysis was originated in the 1960s by A. Robinson. By
considering a hyper-extension\ of real numbers he was able to provide
logically rigorous foundations for infinitely small and infinitely large
numbers.

Nonstandard methods came to Banach space theory from the work of W. A. J.
Luxemburg who introduced the notion of a nonstandard hull. Another approach,
based on the concept of a Banach space ultraproduct, was proposed by J.
Bretagnolle, D. Dacunha-Castelle and J.-L. Krivine. In 1980, B. Maurey \cite%
{Ma} applied the Banach space ultraproduct construction to solve several
difficult problems in metric fixed point theory. His methods have been
extended by numerous authors to obtain a lot of strong results in that
theory (see \cite{AkKh, GoKi, Si}). It seems that nonstandard analysis has
some conceptual advantages over the ultraproduct method because it provides
us with techniques which are not very easy to express in the ultraproduct
setting. But the early approaches to nonstandard analysis appeared too
technical to many mathematicians and required a good background in logic.

At present there exist several interesting frameworks for nonstandard
analysis. In this paper we shall use a modern axiomatic approach based on
Alpha-Theory introduced by V. Benci and M. Di Nasso in \cite{BeDi}. In this
approach, all axioms of ZFC (except foundation) are assumed and for every
set $A$, there exists a set $A^{\ast }$ called the hyper-extension of $A$.
The resulting theory overcomes the distinction between \textquotedblleft
standard\textquotedblright\ and \textquotedblleft
nonstandard\textquotedblright\ objects and is closer to mathematical
practice. Our aim is to signal new possibilities in metric fixed point
theory by applying modern infinitesimal techniques.

Section 2 contains a brief presentation of basic notions including a
nonstandard hull of a Banach space, an intra-convergence of an $^{\ast }%
\mathbb{N}$-sequence and a counterpart of Mazur's lemma. In Section 3 we
formulate a nonstandard iteration scheme for nonexpansive mappings in
uniformly convex spaces. Although it is not clear to what extent nonstandard
analysis can be done constructively, there are recently some attempts to
give infinitesimal analysis computational content (see \cite{Ch, Sa, Se}).
In Section 4 we present a nonstandard approach to fixed-point problems in
direct sums of Banach spaces. We show how to use the notion of intra-convex
sets and a counterpart of Mazur's lemma to improve the results in \cite{PrWi}%
. The reader may compare this approach, closer to the original idea, with
its classical translation in \cite{Wi1}. A brief presentation of
Alpha-Theory is given in the Appendix.

\section{Nonstandard preliminaries}

In this paper we work in the system of nonstandard analysis based on
Alpha-Theory introduced by Benci and Di Nasso in \cite{BeDi}. In this
approach, for every set $A$, there exists a set $^{\ast }A$ called the
hyper-extension (or the star-transform) of $A$, see Appendix A.

The most important for our purposes is the following theorem called the
transfer principle.

\begin{theorem}
\label{Transfer}For every bounded formula $\sigma (x_{1},...,x_{k})$ and for
any sets $a_{1},...,a_{k},$%
\begin{equation*}
\sigma (a_{1},...,a_{k})\Longleftrightarrow \sigma (^{\ast
}a_{1},...,\,^{\ast }a_{k}).
\end{equation*}
\end{theorem}

We will use this theorem several times. See, e.g., \cite{BeDiFo, Gol} for
more details how to apply the transfer principle correctly.

Let $X$ be a real Banach space and let $^{\ast }X$ be its hyper-extension
endowed with a function
\begin{equation*}
^{\ast }\Vert \cdot \Vert :~^{\ast }X\rightarrow ~^{\ast }\mathbb{R}
\end{equation*}%
called an internal norm (or $\ast $-norm) of $^{\ast }X.$ There is a common
practice to omit \textquotedblleft stars\textquotedblright\ when no
confusion can arise and we abbreviate $^{\ast }\Vert \cdot \Vert $ to $\Vert
\cdot \Vert $. Recall that an element $x\in \,^{\ast }X$ is bounded if $%
\Vert x\Vert $ is bounded in $^{\ast }\mathbb{R}$. It is infinitesimal if $%
\Vert x\Vert $ is infinitesimal in $^{\ast }\mathbb{R}$, see Appendix A. Let
$\gal(^{\ast }X)$ denote the set of bounded elements and $\mon(0)$ the set
of infinitesimal elements of $^{\ast }X$. Notice that $\gal(^{\ast }X)$ and $%
\mon(0)$ are vector spaces over $\mathbb{R}$ and we may define $\widetilde{X}
$ as the quotient vector space
\begin{equation*}
\gal(^{\ast }X)/\mon(0).
\end{equation*}%
Let $\pi :\gal(^{\ast }X)\rightarrow \widetilde{X}$ denote the quotient
linear mapping and define a norm on $\widetilde{X}$ by $\Vert y\Vert =\st%
(\Vert x\Vert )$ for all $x\in \gal(^{\ast }X)$, $y=\pi (x)$, where $\st%
(\Vert x\Vert )$ is the standard part of $\Vert x\Vert $ in $\mathbb{R}$.
The vector space $\widetilde{X}$ with the above norm becomes a Banach space
and is called the nonstandard hull of $X$, see, e.g., \cite{FaKe, HeMo, Ng}.
It is clear that $X$ is isometric to a subspace of $\widetilde{X}$ via the
mapping $z\rightarrow $ $\pi (^{\ast }z)$. Virtually, $\pi $ is an extension
of the standard part mapping $\st:\mon(X)\rightarrow X$ and we denote $\pi
(x)$ by $\sh(x)$ or $\sho x$. Thus we have $\sh(x)=\mon(\st(x))$ for every $%
x\in \mon(X)$. We refer to
\begin{equation*}
\sh:\gal(^{\ast }X)\rightarrow \widetilde{X}
\end{equation*}%
as the shadow mapping. Set $^{\circ }A=\left\{ ^{\circ }x:x\in A\right\} $
for any set $A\subset \gal(^{\ast }X)$ and $\widetilde{B}=\,^{\circ }(^{\ast
}B\cap \gal(^{\ast }X))$ for any $B\subset X.$

Let $\mathbb{R}_{+}\,$\ denote the set of positive reals. By an $^{\ast }%
\mathbb{N}$-sequence $(x_{n})_{n\in \,^{\ast }\mathbb{N}}$ in $Y$ we mean a
function $x:$ $^{\ast }\mathbb{N}\rightarrow Y$.

\begin{definition}
An $^{\ast }\mathbb{N}$-sequence $(x_{n})_{n\in \,^{\ast }\mathbb{N}}$ in $%
^{\ast }X$ is said to intra-converge (or $\ast $-converge) to $a\in \,^{\ast
}X$ if%
\begin{equation*}
\forall \varepsilon \in \,^{\ast }\mathbb{R}_{+}\ \exists k\in \,^{\ast }%
\mathbb{N\ \forall }n\in \,^{\ast }\mathbb{N\ }(n\geq k\Rightarrow
\left\Vert x_{n}-a\right\Vert <\varepsilon ).
\end{equation*}
\end{definition}

In a similar way, we can define intra-convergence for the weak topology. Let
$\mathcal{T}$ denote the weak topology on a Banach space $X$.

\begin{definition}
An $^{\ast }\mathbb{N}$-sequence $(x_{n})_{n\in \,^{\ast }\mathbb{N}}$ in $%
^{\ast }X$ is said to weakly intra-converge (or $\ast $-weakly converge) to $%
a\in \,^{\ast }X$ if%
\begin{equation*}
\forall U\in ~^{\ast }\mathcal{T\ }\exists k\in \,^{\ast }\mathbb{N\ \forall
}n\in \,^{\ast }\mathbb{N\ }(a\in U\wedge n\geq k\Rightarrow x_{n}\in U).
\end{equation*}
\end{definition}

Notice that if $(x_{n})$ is a sequence in $X$ converging (resp., weakly
converging) to $x_{0}$, then it follows from transfer that its
hyper-extension $(x_{n})_{n\in \,^{\ast }\mathbb{N}}$ intra-converges
(resp., weakly intra-converges) to $^{\ast }x_{0}$ in $^{\ast }X$.

\begin{definition}
We say that a set $A\subset \,^{\ast }X$ is intra-convex (or $\ast $-convex)
if
\begin{equation*}
\forall \alpha ,\beta \in \,^{\ast }[0,1]\ \mathbb{\forall }x,y\in A\
(\alpha +\beta =1\Rightarrow \alpha x+\beta y\in A).
\end{equation*}
\end{definition}

For $A\subset \,^{\ast }X,$ define%
\begin{equation*}
\conv_{\text{int}}(A)=\bigcup\nolimits_{n\in \,^{\ast }\mathbb{N}}\left\{
\sum_{i=0}^{n}\lambda _{i}x_{i}:\lambda _{i}\in \,^{\ast }[0,1],x_{i}\in
A,0\leq i\leq n,\sum_{i=0}^{n}\lambda _{i}=1\right\} .
\end{equation*}%
The following lemma is a simple application of the transfer principle and
Mazur's lemma.

\begin{lemma}
\label{Lem1}Assume that an internal $^{\ast }\mathbb{N}$-sequence $%
(x_{n})_{n\in \,^{\ast }\mathbb{N}}$ \ in $^{\ast }X$ intra-converges weakly
to $a$. Then $^{\circ }a\in \,^{\circ }\conv_{\text{int}}(\left\{ x_{n}:n\in
\,^{\ast }\mathbb{N}\right\} ).$
\end{lemma}

\begin{proof}
Let $(x_{n})_{n\in \,^{\ast }\mathbb{N}}$ be an internal $^{\ast }\mathbb{N}$%
-sequence in $^{\ast }X$ intra-converging weakly to $a\in \,^{\ast }X$. It
follows from the transfer of Mazur's lemma that for every $\varepsilon \in
\,^{\ast }\mathbb{R}_{+}$ there exists $k\in \,^{\ast }\mathbb{N}$ and $%
\lambda _{0},...,\lambda _{k}\in \,^{\ast }[0,1]$ with $\sum_{i=0}^{k}%
\lambda _{i}=1$ such that $\left\Vert \sum_{i=0}^{k}\lambda
_{i}x_{i}-a\right\Vert \leq \varepsilon .$ Fix a positive $\varepsilon
\simeq 0.$ Then there exists $y\in \conv_{\text{int}}(\left\{ x_{n}:n\in
\,^{\ast }\mathbb{N}\right\} )$ such that $\left\Vert y-a\right\Vert \simeq
0.$ Hence%
\begin{equation*}
^{\circ }a\in \,^{\circ }\conv_{\text{int}}(\left\{ x_{n}:n\in \,^{\ast }%
\mathbb{N}\right\} ).
\end{equation*}
\end{proof}

A routine application of the transfer principle shows that if $A$ is
internal and $\ast $-relatively compact (i.e., for every internal $^{\ast }%
\mathbb{N}$-sequence $(x_{n})_{n\in \,^{\ast }\mathbb{N}}$ of elements in $A$%
, there exists an internal intra-convergent $^{\ast }\mathbb{N}$-subsequence
$(x_{n_{k}})_{k\in \,^{\ast }\mathbb{N}}$), then $\conv_{\text{int}}(A)$ is $%
\ast $-relatively compact, too. We will use this fact together with Lemma %
\ref{Lem1} in Section 4.

\section{Nonstandard Picard Iteration}

Let $(M,\rho )$ be a metric space. An internal mapping $T:\,^{\ast
}M\rightarrow \,^{\ast }M$ is said to be an intra-contraction if there
exists $k\in ~^{\ast }(0,1)$ such that
\begin{equation*}
^{\ast }\rho (Tx,Ty)\leq k\,^{\ast }\rho (x,y)
\end{equation*}%
for all $x,y\in \,^{\ast }M.$

Let $T:\,^{\ast }M\rightarrow \,^{\ast }M$ be an intra-contraction and fix $%
x_{0}\in \,^{\ast }M$. Set $x_{n+1}=Tx_{n}$ for each $n\in \,^{\ast }\mathbb{%
N}.$ Since $T$ is internal, we obtain the $^{\ast }\mathbb{N}$-sequence $%
(T^{n}x_{0})_{n\in \,^{\ast }\mathbb{N}}$ by internal induction.

The following theorem is an internal version of the Banach's Contraction
Principle. We leave its proof to the reader.

\begin{theorem}
\label{contr}Let $(M,\rho )$ be a complete metric space and $T:\,^{\ast
}M\rightarrow \,^{\ast }M$ an intra-contraction. Then $T$ has a unique fixed
point in $^{\ast }M$ and for each $x_{0}\in \,^{\ast }M$ the $^{\ast }%
\mathbb{N}$-sequence $(T^{n}x_{0})_{n\in \,^{\ast }\mathbb{N}}$
intra-converges to this fixed point.
\end{theorem}

Now let $C$ be a nonempty bounded closed and convex subset of a Banach space
$X$ and $T:C\rightarrow C$ a nonexpansive mapping, i.e.,%
\begin{equation*}
\Vert Tx-Ty\Vert \leq \Vert x-y\Vert
\end{equation*}%
for all $x,y\in C$. It is well known that unlike in the case of
contractions, the Picard iteration $(T^{n}x_{0})_{n\in \,\mathbb{N}},$ $%
x_{0}\in C$, may fail to converge. In the last few decades, iterative
methods for finding fixed points of nonexpansive mappings have been studied
extensively. It is worth pointing out two types of such methods. The Mann
iteration is defined by the recursive scheme%
\begin{equation*}
x_{n+1}=(1-\alpha _{n})x_{n}+\alpha _{n}Tx_{n},\ n\in \,\mathbb{N},
\end{equation*}%
where $x_{0}\in C$ and $\alpha _{n}\in \lbrack 0,1].$ The Halpern iteration
is defined by%
\begin{equation*}
x_{n+1}=\alpha _{n}u+(1-\alpha _{n})Tx_{n},\ n\in \,\mathbb{N},
\end{equation*}%
where $x_{0},u\in C$ and $\alpha _{n}\in \lbrack 0,1],\ n\in \,\mathbb{N}$.
Unlike Mann's iteration, a sequence generated by Halpern's scheme is
strongly convergent provided the underlying Banach space is smooth enough
and $(\alpha _{n})$ satisfies some mild conditions. However, in general, the
problem of the convergence of Halpern's iteration is still open even in the
case of uniformly convex spaces. For a deeper discussion of this topic we
refer the reader to \cite{Xu} and the references given there.

New possibilities arises if we consider infinitesimal perturbations of
nonexpansive mappings. Let $T:C\rightarrow C$ be a nonexpansive mapping. By
the transfer principle, we obtain an (intra-nonexpansive) mapping $^{\ast
}T:\,^{\ast }C\rightarrow \,^{\ast }C$ and we can define a nonexpansive
mapping $\widetilde{T}:\widetilde{C}\rightarrow \widetilde{C}$ in the
nonstandard hull of a Banach space $X$ by putting $\widetilde{T}({^{\circ }}%
x)={{}^{\circ }}(^{\ast }Tx)$ for $x\in \,^{\ast }C$. We may regard $C$ as a
subset of $\widetilde{C}$ via the mapping $x\rightarrow $ $^{\circ }(^{\ast
}x)$ and $\widetilde{T}$ as an extension of $T$.

Let $u\in \,^{\ast }C$. Fix a positive infinitesimal $\varepsilon $ and
define%
\begin{equation*}
Sx=(1-\varepsilon )^{\ast }Tx+\varepsilon u,\ x\in \,^{\ast }C.
\end{equation*}%
It is not difficult to check that $S:\,^{\ast }C\rightarrow \,^{\ast }C$ is
an intra-contraction and we can consider for a fixed $x_{0}\in \,^{\ast }C$
a nonstandard Picard iteration%
\begin{equation}
x_{n+1}=S^{n}x_{0},\text{ }n\in \,^{\ast }\mathbb{N}\text{.}  \label{3.1}
\end{equation}%
It follows from Theorem \ref{contr} that the $^{\ast }\mathbb{N}$-sequence $%
(S^{n}x_{0})_{n\in \,^{\ast }\mathbb{N}}$ intra-converges to a point $%
z_{0}\in \,^{\ast }C.$ Notice that $^{\circ }z_{0}\in \widetilde{C}$ is a
fixed point of $\widetilde{T}$ since
\begin{equation*}
\left\Vert ^{\ast }Tz_{0}-z_{0}\right\Vert \leq \varepsilon \simeq 0.
\end{equation*}%
Denote by $P_{C}:\widetilde{X}\rightarrow C$ a metric projection onto $C$:%
\begin{equation*}
P_{C}x=\left\{ y\in C:\left\Vert x-y\right\Vert =\inf_{z\in C}\left\Vert
x-z\right\Vert \right\} .
\end{equation*}%
It is well known that in uniformly convex spaces $P_{C}x$ is a singleton for
every $x\in \widetilde{X}$. Furthermore%
\begin{equation*}
\left\Vert \widetilde{T}P_{C}\,^{\circ }z_{0}-\,^{\circ }z_{0}\right\Vert
=\left\Vert \widetilde{T}P_{C}\,^{\circ }z_{0}-\widetilde{T}\,^{\circ
}z_{0}\right\Vert \leq \left\Vert P_{C}\,^{\circ }z_{0}-\,^{\circ
}z_{0}\right\Vert =\inf_{z\in C}\left\Vert \,^{\circ }z_{0}-z\right\Vert .
\end{equation*}%
But $\widetilde{T}P_{C}\,^{\circ }z_{0}\in C$ and hence $\widetilde{T}%
P_{C}\,^{\circ }z_{0}=P_{C}\,^{\circ }z_{0}$, i.e., $P_{C}\,^{\circ }z_{0}$
is a fixed point of $T$. In this way, we obtain the following theorem.

\begin{theorem}
Let $C$ be a nonempty bounded closed and convex subset of a uniformly convex
Banach space $X$ and $T:C\rightarrow C$ a nonexpansive mapping. Then the
nonstandard Picard iteration given by (\ref{3.1}) intra-converges to a point
$z_{0}\in \,^{\ast }C$. Furthermore, $P_{C}\,^{\circ }z_{0}$ is a fixed
point of $T.$
\end{theorem}

A natural question arises whether the projection $P_{C}$ is at all
necessary, i.e., whether $^{\circ }z_{0}\in C$ if $u,x_{0}\in C$. An
affirmative answer to this question should result in the study of Halpern's
iteration.

\section{Fixed points of direct sums}

Recall that a Banach space $X$ is said to have the fixed point property
(FPP) if every nonexpansive self-mapping defined on a nonempty bounded
closed and convex set $C\subset X$ has a fixed point. A Banach space $X$ is
said to have the weak fixed point property (WFPP) if the additional
assumption is added that $C$ is weakly compact.

The problem of whether FPP or WFPP is preserved under direct sum of Banach
spaces is an old one. In 1968, L. P. Belluce, W. A. Kirk and E. F. Steiner
\cite{BKS} proved that the direct sum of two Banach spaces with normal
structure, endowed with the \textquotedblleft maximum\textquotedblright\
norm, also has normal structure. Since then, the preservation of normal
structure and conditions which guarantee normal structure have been studied
extensively and the problem is now quite well understood (see \cite{DhSa}
for a survey). But the situation is much more difficult if at least one of
these spaces lacks weak normal structure. We note here the results of S.
Dhompongsa, A. Kaewcharoen and A. Kaewkhao \cite{DKK}, and M. Kato and T.
Tamura (see \cite{KaTa1, KaTa2}).

Recently, a few general fixed point theorems in direct sums were proved in
\cite{Wi1, Wi2} (see also \cite{PrWi, Wi}). Although their proofs were
formulated in standard terms, the original ideas came from nonstandard
analysis. In this section we present the original proof of the main result
in \cite{Wi1} which is, in our opinion, more insightful than its classical
translation.

Let us first recall terminology concerning direct sums. A norm $\left\Vert
\cdot \right\Vert $ on $\mathbb{R}^{2}$ is said to be monotone if
\begin{equation*}
\Vert (x_{1},y_{1})\Vert \leq \Vert (x_{2},y_{2})\Vert \ \ \ \text{whenever}%
\ \ 0\leq x_{1}\leq x_{2},0\leq y_{1}\leq y_{2}.
\end{equation*}%
A norm $\left\Vert \cdot \right\Vert $ is said to be strictly monotone if
\begin{align*}
\ \Vert (x_{1},y_{1})\Vert <\Vert (x_{2},y_{2})\Vert \ \ \ \text{whenever}\
\ & 0\leq x_{1}\leq x_{2},0\leq y_{1}<y_{2}\ \  \\
\text{or}\ \ & 0\leq x_{1}<x_{2},0\leq y_{1}\leq y_{2}.
\end{align*}%
It is easy to see that $\ell _{p}^{2}$-norms, $1\leq p<\infty ,$ are
strictly monotone. We will assume that the norm is normalized, i.e.,
\begin{equation*}
\Vert (1,0)\Vert =...=\Vert (0,1)\Vert =1.
\end{equation*}%
F. F. Bonsall and J. Duncan \cite{BoDu} showed that the set of all monotone
and normalized norms on $\mathbb{R}^{2}$ is in one-to-one correspondence
with the set $\Psi $ of all continuous convex functions on $[0,1]$
satisfying $\psi (0)=\psi (1)=1$ and $\max \{1-t,t\}\leq \psi (t)\leq 1$ for
$0\leq t\leq 1,$ where the correspondence is given by
\begin{equation}
\psi (t)=\left\Vert (1-t,t)\right\Vert ,\ 0\leq t\leq 1.  \label{n1}
\end{equation}%
Conversely, for any $\psi \in \Psi $ define%
\begin{equation*}
\Vert (x_{1},x_{2})\Vert _{\psi }=(\left\vert x_{1}\right\vert +\left\vert
x_{2}\right\vert )\psi (\left\vert x_{2}\right\vert /\left\vert
x_{1}\right\vert +\left\vert x_{2}\right\vert )
\end{equation*}%
for $(x_{1},x_{2})\neq (0,0)$ and $\Vert (0,0)\Vert _{\psi }=0.$ Then $\Vert
\cdot \Vert _{\psi }$ is an absolute and normalized norm which satisfies (%
\ref{n1}). It was proved in \cite[Corollary 3]{TaKaSa2} that a norm $\Vert
\cdot \Vert _{\psi }$ in $\mathbb{R}^{2}$ is normalized and strictly
monotone iff
\begin{equation*}
\psi (t)>\psi _{\infty }(t)
\end{equation*}%
for all $0<t<1.$ Let $X,Y$ be Banach spaces and $\psi \in \Psi $. We shall
write $X\oplus _{\psi }Y$ for the $\psi $-direct sum of $X,Y$ with the norm $%
\Vert (x,y)\Vert _{\psi }=\Vert (\Vert x\Vert ,\Vert y\Vert )\Vert _{\psi }$%
, where $(x,y)\in X\times Y$.

A Banach space $X$ is said to have the generalized Gossez-Lami Dozo property
(GGLD, in short) if
\begin{equation*}
\limsup_{m\rightarrow \infty }\limsup_{n\rightarrow \infty }\Vert
x_{n}-x_{m}\Vert >1
\end{equation*}%
whenever $(x_{n})$ converges weakly to $0$ and $\lim_{n\rightarrow \infty
}\Vert x_{n}\Vert =1.$ It is known that the GGLD property is weaker than
weak uniform normal structure (see, e.g., \cite{SiSm}).

The following lemma was proved in \cite[Lemma 4]{PrWi} (see also \cite{GaLl,
SiSm}).

\begin{lemma}
\label{y=0} Let $X\oplus _{\psi }Y$ be a $\psi $-direct sum of Banach spaces
$X$, $Y$ with a strictly monotone norm. Assume that $Y$ has the GGLD
property, the vectors $w_{n}=(x_{n},y_{n})\in X\oplus _{\psi }Y$ tend weakly
to 0 and
\begin{equation*}
\lim_{n,m\rightarrow \infty ,n\neq m}\Vert w_{n}-w_{m}\Vert _{\psi
}=\lim_{n\rightarrow \infty }\Vert w_{n}\Vert _{\psi }.
\end{equation*}%
Then $\lim_{n\rightarrow \infty }\Vert y_{n}\Vert =0$.
\end{lemma}

We are now in a position to give a nonstandard proof of the following
theorem.

\begin{theorem}[\protect\cite{Wi1}]
\label{Th1}Let $X$ be a Banach space with WFPP and suppose $Y$ has the GGLD
property. Then $X\oplus _{\psi }Y$ with a strictly monotone norm has WFPP.
\end{theorem}

\begin{proof}
The proof will be divided into 5 steps.\smallskip

Step 1. We follow the classical arguments in metric fixed point theory.
Assume that $X\oplus _{\psi }Y$ does not have WFPP. Then, there exist a
weakly compact convex subset $C$ of $X\oplus _{\psi }Y$ and a nonexpansive
mapping $T:C\rightarrow C$ without a fixed point. By the Kuratowski-Zorn
lemma, there exists a convex and weakly compact set $K\subset C$ which is
minimal invariant under $T$ and which is not a singleton. Let $%
(w_{n})=((x_{n}^{\prime },y_{n}^{\prime }))$ be an approximate fixed point
sequence for $T$ in $K$, i.e., $\lim_{n\rightarrow \infty }\Vert
Tw_{n}-w_{n}\Vert _{\psi }=0$. Without loss of generality we can assume that
$\diam K=1$, $(w_{n})$ converges weakly to $(0,0)\in K$ and the double limit
$\lim_{n,m\rightarrow \infty ,n\neq m}\Vert w_{n}-w_{m}\Vert _{\psi }$
exists. It follows from the Goebel-Karlovitz lemma (see \cite{Go, Ka}) that
\begin{equation}
\lim_{n,m\rightarrow \infty ,n\neq m}\Vert w_{n}-w_{m}\Vert _{\psi
}=1=\lim_{n\rightarrow \infty }\Vert w_{n}-w\Vert _{\psi }  \label{dseq}
\end{equation}%
for every $w\in K.$ Hence $\lim_{n\rightarrow \infty }\Vert y_{n}^{\prime
}\Vert =0$ by Lemma \ref{y=0}.\smallskip

Step 2. Let $(w_{n})_{n\in \,^{\ast }\mathbb{N}}$ be a hyper-extension of
the sequence $(w_{n})_{n\in \mathbb{N}}.$ Since $(\mathbb{R}^{2},\left\Vert
\cdot \right\Vert _{\psi })$ is a finite dimensional space, the norm $%
\left\Vert \cdot \right\Vert _{\psi }$ is strictly monotone iff it is
uniformly monotone. It follows, using transfer, that for every $\varepsilon
\in \,^{\ast }\mathbb{R}_{+},$ there exists $\delta (\varepsilon )\in
\,^{\ast }\mathbb{R}_{+}$ such that if $(\bar{a},\bar{b}),(\bar{a},\bar{c})$
belong to $^{\ast }B_{(\mathbb{R}^{2},\left\Vert \cdot \right\Vert _{\psi
})} $ and $\left\Vert (\bar{a},\bar{b})\right\Vert _{\psi }<\left\Vert (\bar{%
a},\bar{c})\right\Vert _{\psi }+\delta (\varepsilon ),$ then $\left\Vert
\bar{b}\right\Vert <\left\Vert \bar{c}\right\Vert +\varepsilon .$ Fix an
unbounded $\omega \in \,^{\ast }\mathbb{N}$ and put $\eta =\frac{1}{\omega }%
\simeq 0.$ Let
\begin{equation*}
\varepsilon _{i}=\min \{\eta \delta (\eta ^{i})/3,\eta ^{i+1}\},\ \ i\in
\,^{\ast }\mathbb{N}.
\end{equation*}%
By transfer, $\Vert ^{\ast }Tw_{n}-w_{n}\Vert _{\psi }$ and $\Vert
y_{n}^{\prime }\Vert $ intra-converge to $0$ and hence we can fix $%
v_{0}=w_{n_{0}}=(x_{0},y_{0})$ such that $\Vert ^{\ast }Tv_{0}-v_{0}\Vert
_{\psi }<\varepsilon _{0}$ and $\Vert y_{0}\Vert <\varepsilon _{0}.$ For
hypernatural numbers $1\leq j\leq \omega ,$ write $D_{j}^{0}=\left\{
v_{0}\right\} .$ We shall define an internal $^{\ast }\mathbb{N}$%
-subsequence $(v_{n})_{n\in \,^{\ast }\mathbb{N}}$ of $(w_{n})_{n\in
\,^{\ast }\mathbb{N}}$ and an internal family $\left\{ D_{j}^{i}\right\}
_{1\leq j\leq \omega ,i\in \,^{\ast }\mathbb{N}}$ of $\ast $-relatively
compact subsets of $^{\ast }K$ by internal induction. Choose $%
v_{1}=w_{n_{1}}=(x_{1},y_{1})$, $n_{0}<n_{1}\in $ $^{\ast }\mathbb{N}$ in
such a way that $\Vert ^{\ast }Tv_{1}-v_{1}\Vert _{\psi }<\varepsilon _{1}$,
$\Vert y_{1}\Vert <\varepsilon _{1}$ and $\Vert v_{1}-v_{0}\Vert _{\psi
}>1-\varepsilon _{1}$ (notice that, by transfer of (\ref{dseq}), $\Vert
w_{n}-v_{0}\Vert $ intra-converges to $1$). Let us put
\begin{equation*}
D_{1}^{1}=\conv_{\text{int}}\left\{ v_{0},v_{1}\right\}
\end{equation*}%
and
\begin{equation*}
D_{j+1}^{1}=\conv_{\text{int}}(D_{j}^{1}\cup \,^{\ast }T(D_{j}^{1}))
\end{equation*}%
for $1\leq j<\omega $. By internal induction, $\left\{
D_{1}^{1},...,D_{\omega }^{1}\right\} $ is a well-defined internal family of
$\ast $-relatively compact subsets of $^{\ast }K$ with $D_{1}^{1}\subset
...\subset D_{\omega }^{1}.$

Now suppose that we have chosen an internal $k$-tuple $n_{1}<...<n_{k}$ $(k$
$\in \,^{\ast }\mathbb{N\smallsetminus }\{0\},n_{1}>n_{0}),$ $%
v_{i}=w_{n_{i}}=(x_{i},y_{i}),0\leq i\leq k,$ and internal $k$-tuple $%
(\left\{ D_{1}^{i},...,D_{\omega }^{i}\right\} )_{1\leq i\leq k}$ of subsets
of $^{\ast }K$ such that for each $i\in \left\{ 1,...,k\right\} $ and $j\in
\{1,...,\omega -1\}:$

\begin{enumerate}
\item[(i)] $\Vert ^{\ast }Tv_{i}-v_{i}\Vert _{\psi }<\varepsilon _{i},$

\item[(ii)] $\Vert y_{i}\Vert <\varepsilon _{i},$

\item[(iii)] $\Vert v_{i}-v\Vert _{\psi }>1-\varepsilon _{i}$ for all $v\in
D_{\omega }^{i-1},$

\item[(iv)] $D_{1}^{i}=\conv_{\text{int}}(D_{1}^{i-1}\cup \left\{
v_{i}\right\} ),$

\item[(v)] $D_{j+1}^{i}=\conv_{\text{int}}(D_{j}^{i}\cup \,^{\ast
}T(D_{j}^{i})).$
\end{enumerate}

Then, there exist (internally chosen) $n_{k+1}>n_{k}$, $%
v_{k+1}=w_{n_{k+1}}=(x_{k+1},y_{k+1})$ such that $\Vert ^{\ast
}Tv_{k+1}-v_{k+1}\Vert <\varepsilon _{k+1}$, $\Vert y_{k+1}\Vert
<\varepsilon _{k+1}$ and $\Vert v_{k+1}-v\Vert >1-\varepsilon _{k+1}$ for
all $v\in D_{\omega }^{k}$ (the last inequality follows from the $\ast $%
-relative compactness of $D_{\omega }^{k}$). Let us put
\begin{equation*}
D_{1}^{k+1}=\conv_{\text{int}}(D_{1}^{k}\cup \left\{ v_{k+1}\right\} )
\end{equation*}%
and
\begin{equation*}
D_{j+1}^{k+1}=\conv_{\text{int}}(D_{j}^{k+1}\cup \,^{\ast }T(D_{j}^{k+1}))
\end{equation*}%
for $1\leq j<\omega $. Then, by internal induction on $j$, $\left\{
D_{1}^{k+1},...,D_{\omega }^{k+1}\right\} $ is a well-defined internal
family of $\ast $-relatively compact subsets of $^{\ast }K.$ Hence, by
internal induction on $i$, we obtain an internal sequence $(v_{n})_{n\in
\,^{\ast }\mathbb{N}}$ and an internal family of sets $\left\{
D_{j}^{i}\right\} _{1\leq j\leq \omega ,i\in \,^{\ast }\mathbb{N}}$ such
that (i)-(v) are satisfied for every $j\in \{1,...,\omega -1\}$ and $i\in
\,^{\ast }\mathbb{N\smallsetminus }\{0\}.$\smallskip

Step 3. We claim that for every $1\leq j\leq \omega $, $i\in \,^{\ast }%
\mathbb{N\smallsetminus }\{0\}$ and $u\in D_{j}^{i+1}$ there exists $v\in
D_{j}^{i}$ such that
\begin{equation}
\left\Vert v-u\right\Vert _{\psi }+\left\Vert u-v_{i+1}\right\Vert _{\psi
}\leq \left\Vert v-v_{i+1}\right\Vert _{\psi }+3(j-1)\varepsilon _{i+1}.
\label{step3}
\end{equation}%
Fix $i\in \,^{\ast }\mathbb{N\smallsetminus }\{0\}.$ We shall proceed by
internal induction with respect to $j$. For $j=1$ and $u\in D_{1}^{i+1}=\conv%
_{\text{int }}(D_{1}^{i}\cup \left\{ v_{i+1}\right\} )$ there exists $v\in
D_{1}^{i}$ such that
\begin{equation*}
\left\Vert v-u\right\Vert _{\psi }+\left\Vert u-v_{i+1}\right\Vert _{\psi
}=\left\Vert v-v_{i+1}\right\Vert _{\psi }.
\end{equation*}

Fix $1\leq j<\eta $ and suppose that for every $u\in D_{j}^{i+1}$ there
exists $v\in D_{j}^{i}$ such that (\ref{step3}) is satisfied. Let $u\in
D_{j+1}^{i+1}=$ $\conv_{\text{int}}(D_{j}^{i+1}\cup ~^{\ast
}T(D_{j}^{i+1})). $ The inductive step is obvious if $u\in D_{j}^{i+1}$ so
take $u\in \,^{\ast }T(D_{j}^{i+1})$. Then $u=\,^{\ast }T\bar{u}$ for some $%
\bar{u}\in D_{j}^{i+1}$ and, by assumption, there exists $\bar{v}\in
D_{j}^{i}$ such that
\begin{equation*}
\left\Vert \bar{v}-\bar{u}\right\Vert _{\psi }+\left\Vert \bar{u}%
-v_{i+1}\right\Vert _{\psi }\leq \left\Vert \bar{v}-v_{i+1}\right\Vert
_{\psi }+3(j-1)\varepsilon _{i+1}.
\end{equation*}%
Let $v=~^{\ast }T\bar{v}\in D_{j+1}^{i}$. Then
\begin{align}
& \left\Vert v-u\right\Vert _{\psi }+\left\Vert u-v_{i+1}\right\Vert _{\psi
}\leq \left\Vert \bar{v}-\bar{u}\right\Vert _{\psi }+\left\Vert \bar{u}%
-v_{i+1}\right\Vert _{\psi }+\left\Vert ^{\ast }Tv_{i+1}-v_{i+1}\right\Vert
_{\psi }  \notag \\
& \leq \left\Vert \bar{v}-v_{i+1}\right\Vert _{\psi }+(3j-2)\varepsilon
_{i+1}<\left\Vert v-v_{i+1}\right\Vert _{\psi }+(3j-1)\varepsilon _{i+1},
\label{step3b}
\end{align}%
since $\left\Vert ^{\ast }Tv_{i+1}-v_{i+1}\right\Vert _{\psi }<\varepsilon
_{i+1}$ and $\left\Vert v-v_{i+1}\right\Vert _{\psi }>1-\varepsilon
_{i+1}\geq \left\Vert \bar{v}-v_{i+1}\right\Vert _{\psi }-\varepsilon
_{i+1}. $

Now let $u=\sum\nolimits_{s=1}^{t}\lambda _{s}u_{s}$ for some $u_{s}\in
D_{j}^{i+1}\cup ~^{\ast }T(D_{j}^{i+1}),\lambda _{s}\in \,^{\ast }\left[ 0,1%
\right] ,1\leq s\leq t\in \,^{\ast }\mathbb{N}$, $\sum\nolimits_{s=1}^{t}%
\lambda _{s}=1.$ Then, by (\ref{step3}) and (\ref{step3b}), there exist $%
\bar{v}_{1},...,\bar{v}_{t}\in D_{j+1}^{i}$ such that%
\begin{equation*}
\left\Vert \bar{v}_{s}-u_{s}\right\Vert _{\psi }+\left\Vert
u_{s}-v_{i+1}\right\Vert _{\psi }\leq \left\Vert \bar{v}_{s}-v_{i+1}\right%
\Vert _{\psi }+(3j-1)\varepsilon _{i+1},1\leq s\leq t.
\end{equation*}%
Hence%
\begin{align*}
& \left\Vert \sum\nolimits_{s=1}^{t}\lambda _{s}\bar{v}_{s}-u\right\Vert
_{\psi }+\left\Vert u-v_{i+1}\right\Vert _{\psi }\leq
\sum\nolimits_{s=1}^{t}\lambda _{s}\left\Vert \bar{v}_{s}-v_{i+1}\right\Vert
_{\psi }+(3j-1)\varepsilon _{i+1} \\
& \leq 1+(3j-1)\varepsilon _{i+1}<\left\Vert \sum\nolimits_{s=1}^{t}\lambda
_{s}\bar{v}_{s}-v_{i+1}\right\Vert _{\psi }+3j\varepsilon _{i+1},
\end{align*}%
since, by (iii), $\dist(D_{\omega }^{i},v_{i+1})>1-\varepsilon _{i+1},$ and
the claim is proved.\smallskip

Step 4. Let $1\leq j\leq \omega $, $i\in \,^{\ast }\mathbb{N}$ and $%
u=(a,b)\in D_{j}^{i}.$ We claim that $\left\Vert b\right\Vert $ $\simeq 0.$
Fix $i\geq 2.$ By Step 3, take $v=(x,y)\in D_{j}^{i-1}$ such that
\begin{equation*}
\left\Vert v-u\right\Vert _{\psi }+\left\Vert u-v_{i}\right\Vert _{\psi
}\leq \left\Vert v-v_{i}\right\Vert _{\psi }+3(j-1)\varepsilon
_{i}<\left\Vert v-v_{i}\right\Vert _{\psi }+3\omega \varepsilon _{i}.
\end{equation*}%
Hence
\begin{equation*}
\left\Vert (\left\Vert x-x_{i}\right\Vert ,\left\Vert y-b\right\Vert
+\left\Vert b-y_{i}\right\Vert )\right\Vert _{\psi }<\left\Vert (\left\Vert
x-x_{i}\right\Vert ,\left\Vert y-y_{i}\right\Vert )\right\Vert _{\psi
}+3\omega \varepsilon _{i}
\end{equation*}%
which yields
\begin{equation*}
\left\Vert y-b\right\Vert +\left\Vert b-y_{i}\right\Vert <\left\Vert
y-y_{i}\right\Vert +\eta ^{i}
\end{equation*}%
since $3\omega \varepsilon _{i}\leq \delta (\eta ^{i}).$ Consequently,
\begin{equation*}
\left\Vert b\right\Vert <\left\Vert y\right\Vert +\left\Vert
y_{i}\right\Vert +\frac{1}{2}\eta ^{i}.
\end{equation*}%
Repeating this procedure $(i-1)$ times we obtain by internal induction an
element $(\bar{x},\bar{y})\in D_{\omega }^{1}$ such that%
\begin{equation*}
\left\Vert b\right\Vert <\left\Vert \bar{y}\right\Vert +\left\Vert
y_{2}\right\Vert +\frac{1}{2}\eta ^{2}+...+\left\Vert y_{i}\right\Vert +%
\frac{1}{2}\eta ^{i}.
\end{equation*}%
Furthermore, it is not difficult to show that $\left\Vert \bar{y}\right\Vert
<\omega \varepsilon _{1}$ (see \cite[Lemma 3.1]{Wi2}). Hence $\left\Vert
b\right\Vert <\omega \varepsilon _{1}+(\varepsilon _{2}+...+\varepsilon
_{i})+\frac{1}{2}(\eta ^{2}+...+\eta ^{i})<\eta +2\eta ^{3}+\eta ^{2}\simeq
0.$\smallskip

Step 5. Let $D_{j}=\bigcup\nolimits_{i\in \,^{\ast }\mathbb{N}}D_{j}^{i}$
for $1\leq j\leq \omega $. Then we can easily prove that $D_{1}\subset
D_{2}\subset ...\subset D_{\omega }$ and $^{\ast }T(D_{j})\subset D_{j+1}$
for $1\leq j<\omega $. Moreover, a sequence $(v_{n})_{n\in ~^{\ast }\mathbb{N%
}}$ intra-converges to $(0,0)$ and hence, by Lemma \ref{Lem1}, $^{\circ
}(0,0)\in ~^{\circ }D_{1}.$ Let%
\begin{equation*}
D=\cl(\bigcup\nolimits_{j\in \mathbb{N\setminus }\left\{ 0\right\}
}{}^{\circ }D_{j}).
\end{equation*}

Notice that $D$ is closed and convex subset of $\widetilde{K}$ which is
invariant under $\widetilde{T}$. Moreover, $^{\circ }(0,0)$ $\in D$ and
consequently the set $M=D\cap \left\{ ^{\circ }(^{\ast }x):x\in K\right\} $
is nonempty, closed, convex and $\widetilde{T}$-invariant. It follows from
Step 4 that $M\subset \left\{ ^{\circ }(^{\ast }x):x\in X\right\} \times
\left\{ 0\right\} $ and therefore $M$ is isometric to a subset of $X$. Since
$X$ has WFPP, $\widetilde{T}$ has a fixed point in $M$, which contradicts
our assumption.
\end{proof}

\appendix

\section{Alpha-Theory}

At present there exist several frameworks for nonstandard analysis. In this
paper we use an axiomatic approach introduced in \cite{BeDi} (see also the
related system $^{\ast }$ZFC \cite{Di1}). This approach is based on the
existence of a new mathematical object $\alpha $ which can be seen as a new
\textquotedblleft ideal\textquotedblright\ number added to $\mathbb{N}$. Our
exposition follows \cite[Sect. 8.3d]{KaRe} (we do not assume the existence
of atoms).

The Alpha-Theory is a theory in the language $\mathcal{L}^{\prime }=\left\{
\in ,J\right\} $ of set theory extended by a new binary relation symbol $J$.
The axioms include all of ZFC minus Regularity, together with the following
five axioms: \smallskip

\begin{enumerate}
\item[\textbf{\ J1}.] $J$ is a function defined on the class of all
sequences of arbitrary sets, i.e.,%
\begin{equation*}
\forall \varphi (%
\Seq%
(\varphi )\Rightarrow \exists !xJ(\varphi ,x))\wedge \forall \varphi \forall
x(J(\varphi ,x)\Rightarrow \Seq(\varphi )),
\end{equation*}%
where $\Seq(\varphi )$ means that $\varphi $ is a sequence, i.e., a function
with the domain $\mathbb{N}$.
\end{enumerate}

Let $J( \varphi ) $ be the unique $x$ which satisfies $J( \varphi ,x) .$

\begin{enumerate}
\item[ \textbf{J2}.] If $f$ is a function defined on a set $A$ and $\varphi
,\psi :\mathbb{N}\rightarrow A$, then $J( \varphi ) =J( \psi ) $ implies $J(
f\circ \varphi ) =J( f\circ \psi ) .$

\item[ \textbf{J3}.] $J( c_{m}) =m$ for any natural $m$, where $c_{m}( n) =m$
for all $n\in \mathbb{N},$

\item[ \ \ ] $J( \text{id}) \notin \mathbb{N}$, where id$( n) =n$ for all $%
n\in \mathbb{N}$.

\item[\textbf{J4}.] If $\vartheta ( n) =\left\{ \varphi ( n) ,\psi ( n)
\right\} $ for all $n\in \mathbb{N}$, then $J( \vartheta ) =\left\{ J(
\varphi ) ,J( \psi ) \right\} .$

\item[\textbf{J5}.] For any $\varphi $, $J(\varphi )=\{J(\psi ):\psi (n)\in
\varphi (n)\}$ for all $n\in \mathbb{N}.$
\end{enumerate}

\smallskip

Let us define $^{\ast }x=J(c_{x})$ for any set $x$ (where $c_{x}(n)=x$ for
all $n\in \mathbb{N}$). Put
\begin{equation*}
\alpha =J(\text{id}).
\end{equation*}%
By the axiom J3, $\alpha \notin \mathbb{N}$ and, by J5, $\alpha \in \,^{\ast
}\mathbb{N}$. It turns out (see \cite[Prop. 2.3]{BeDi}) that if $%
f:A\rightarrow B$, then $^{\ast }f$ is a function from $^{\ast }A$ to $%
^{\ast }B$ and $^{\ast }f(J(\varphi ))=J(f\circ \varphi )$ for any $\varphi :%
\mathbb{N}\rightarrow A$. Taking $\varphi =\ $id and $f=\varphi $ we obtain $%
^{\ast }\varphi (\alpha )=J(\varphi )$ for any sequence $\varphi $. Thus, $J$%
-extensions are simply values of the $\ast $-extended functions at a
\textquotedblleft non-standard natural number\textquotedblright\ $\alpha $.

A set $x$ is said to be internal if there exists a sequence $\varphi $ such
that $x=J( \varphi ) $. Equivalently, $x$ is internal if there exists $y$
such that $x\in $ $^{\ast }y$. A set $x$ is external if it is not internal.

One of the fundamental tools in nonstandard analysis is the transfer
principle which is an application of a famous theorem of \L o\'{s}. Recall
that a formula $\sigma $ is bounded if it is constructed from atomic
formulae using connectives and bounded quantifiers $\forall x\in y$ (i.e., $%
\forall x\ x\in y\Rightarrow ...$), $\exists x\in y$ (i.e., $\exists x\ x\in
y\wedge ...$). The following theorem (see \cite[Th. 6.2]{BeDi}, \cite[Cor.
8.3.13]{KaRe}) shows that the transfer principle is satisfied in
Alpha-Theory.

\begin{theorem}
\label{TransferA}For every bounded formula $\sigma (x_{1},...,x_{k})$ in the
first-order language $\mathcal{L}=\left\{ \in \right\} $ and for any sets $%
a_{1},...,a_{k},$%
\begin{equation*}
\sigma (a_{1},...,a_{k})\Longleftrightarrow \sigma (^{\ast
}a_{1},...,\,^{\ast }a_{k}).
\end{equation*}
\end{theorem}

The following useful theorem (known as the Internal Definition Principle) is
a rather straightforward consequence of the transfer principle.

\begin{theorem}
If $\sigma ( x,x_{1},...,x_{k}) $ is a bounded formula in the first-order
language $\mathcal{L}=\left\{ \in \right\} $ and $b,b_{1},...,b_{k}$ are
internal sets, then $\left\{ x\in b:\sigma ( x,b_{1},...,b_{k}) \right\} $
is an internal set.
\end{theorem}

Another notion which is frequently used in nonstandard analysis is the
so-called countable saturation.

\begin{theorem}[{see \protect\cite[Th. 4.4]{BeDi}}]
Let $\left\{ A_{n}:n\in \mathbb{N}\right\} $ be a countable family of
internal sets with the finite intersection property. Then the intersection $%
\bigcap\nolimits_{n\in \mathbb{N}}A_{n}\neq \emptyset $.
\end{theorem}

Let $(\mathbb{R},+,\cdot ,\leq )$ be the complete ordered field of real
numbers. Then, by transfer, we obtain an ordered field\ $(^{\ast }\mathbb{R}%
,^{\ast }+,^{\ast }\cdot ,^{\ast }\leq )$. There is a common practice to
omit \textquotedblleft stars\textquotedblright\ when no confusion can arise.
Notice that $^{\ast }\mathbb{R=}$ $\left\{ \varphi (\alpha ):\varphi :%
\mathbb{N}\rightarrow \mathbb{R}\right\} $ and hence $\left\{ ^{\ast }x:x\in
\mathbb{R}\right\} \subset $ $^{\ast }\mathbb{R}$. Although, in general, $%
x\neq \,^{\ast }x$ we do not usually distinguish between $x$ and $^{\ast }x$
and regard the set of reals as a subset of $^{\ast }\mathbb{R}$. Elements of
$^{\ast }\mathbb{R}$ are called hyperreals.

\begin{definition}
A hyperreal number $x$ is said to be

\begin{enumerate}
\item[(i)] bounded if $x=O( 1) $, i.e., $\left\vert x\right\vert \leq c$ for
some $c\in \mathbb{R},$

\item[(ii)] infinitesimal if $x=o( 1) $, i.e., $\left\vert x\right\vert \leq
\varepsilon $ for every positive $\varepsilon \in \mathbb{R}$,

\item[(iii)] unbounded if $1/x=o( 1) $.
\end{enumerate}
\end{definition}

Notice that $\alpha >n$ for every $n\in \mathbb{N}$ (see \cite[Prop. 2.5]%
{BeDi}) and hence $1/\alpha $ is an example of a (nonzero) infinitesimal. We
say that $x$ and $y$ are infinitely close, denoted by $x\simeq y$, if $x-y$
is infinitesimal. This defines an equivalence relation on\ $^{\ast }\mathbb{R%
}$ and the monad (or the halo) of $x$ is the equivalence class
\begin{equation*}
\mon(x)=\left\{ y\in \,^{\ast }\mathbb{R}:x\simeq y\right\} .
\end{equation*}%
We say that $x$ and $y$ are of bounded distance apart, denoted by $x\sim y,$
if $x-y$ is bounded. The galaxy of $x$ is the equivalence class
\begin{equation*}
\gal(x)=\left\{ y\in \,^{\ast }\mathbb{R}:x\sim y\right\}
\end{equation*}%
(see \cite{Gol, LoWo} and references therein). If a hyperreal $x$ is
bounded, i.e., $x\in \gal(0)$, the unique $a\in \mathbb{R}$ such that $%
x\simeq a$ is called the standard part of $x$ and is denoted by $\st(x)$.
These notions can be generalized in the following way. Let $X$ be a real
Banach space and let $^{\ast }X$ be its hyper-extension endowed with a
function $^{\ast }\Vert \cdot \Vert :~^{\ast }X\rightarrow ~^{\ast }\mathbb{R%
}$ called an internal norm (or $\ast $-norm) of $^{\ast }X$. By transfer, $%
^{\ast }\Vert \cdot \Vert $ is homogeneous over $^{\ast }\mathbb{R}$ and
satisfies the triangle inequality. As before, we do not distinguish between $%
x$ and $^{\ast }x$, and abbreviate $^{\ast }\Vert \cdot \Vert $ to $\Vert
\cdot \Vert $. The monad of $x\in \,^{\ast }X$ is the equivalence class $\mon%
(^{\ast }X,x)=\left\{ y\in \,^{\ast }X:\left\Vert x-y\right\Vert \simeq
0\right\} $ ($\mon(x)$ for brevity) and the galaxy of $x$ is the equivalence
class $\gal(^{\ast }X,x)=\left\{ y\in \,^{\ast }X:\left\Vert x-y\right\Vert
\sim 0\right\} .$ The set $\gal(^{\ast }X,0)$ is called the principal galaxy
and denoted by $\gal(^{\ast }X)$. Let $\mon(X)=\bigcup_{x\in X}\mon(x).$
Notice that in general $\mon(X)$ is a proper subset of $\gal(^{\ast }X)$. If
$x\in \mon(X)$, the unique $a\in X$ such that $\left\Vert x-a\right\Vert
\simeq 0$ is called the standard part of $x$ and is denoted, as in a real
case, by $\st(x).$ We refer to $\st:\mon(X)\rightarrow X$ as the standard
part mapping.

It was proved in \cite[Th. 6.4]{BeDi}, that ZFC is
faithfully interpretable in the Alpha-Theory, i.e., a sentence $\sigma $ in
the language $\mathcal{L}=\left\{ \in \right\} $ is a theorem of ZFC if and
only if its relativization $\sigma ^{WF}$ to the class of well-founded sets
is a theorem of the Alpha-Theory. In other words, the Alpha Theory proves
those and only those statements ($\in $-statements, to be precise) about
well-founded sets which ZFC proves about all sets.

\end{document}